\documentclass[a4paper,fleqn]{cas-sc}

\usepackage[authoryear,longnamesfirst]{natbib}
\usepackage{times}
\usepackage{amssymb,amsthm,latexsym,amsmath,epsfig,pgf,tikz,multirow}
\usepackage{lineno,booktabs}
\usepackage{pdflscape}
\usetikzlibrary{shapes.geometric,calc,decorations.pathreplacing}
\usetikzlibrary{arrows.meta}
\usepackage{hyperref}
\hypersetup{
	colorlinks=true,
	linkcolor=blue,
	filecolor=magenta,      
	urlcolor=cyan,
	pdftitle={Overleaf Example},
	pdfpagemode=FullScreen,
}

\newtheorem{theorem}{Theorem}[section]
\newtheorem*{theorem A}{Theorem A}
\newtheorem*{theorem B}{N\"olker's Theorem}
\newtheorem{lemma}{Lemma}[section]

\newtheorem{corollary}{Corollary}[section]

\newtheorem {conjecture}{Conjecture}
\theoremstyle{remark}

\theoremstyle{remark}

\theoremstyle{definition}
\newtheorem{definition}{Definition}
\def\tsc#1{\csdef{#1}{\textsc{\lowercase{#1}}\xspace}}
\tsc{WGM}
\tsc{QE}
\newcommand{\T}[1]{T_{(#1)}}
\newcommand{\G}{\mathcal{G}}

\begin{document}
\let\WriteBookmarks\relax
\def\floatpagepagefraction{1}
\def\textpagefraction{.001}

\shorttitle{Grundy Values of Node-Kayles on Level-Regular Trees}    

\shortauthors{Nuttanon Songsuwan}  

\title [mode = title]{Grundy Values of Node-Kayles on Level-Regular Trees}  



%

\author[1,2]{Nuttanon Songsuwan}[orcid=0000-0002-6737-7561]

\cormark[1]


\ead{nuttanon.son@ku.th}



\affiliation[1]{organization={Faculty of Science at Sriracha, Kasetsart University, Sriracha Campus},
            country={Thailand}}
\affiliation[2]{organization={VIIL: Vision \& Indochina Intelligence Lab (The EEC AI Hub)},
	country={Thailand}}






\cortext[1]{Corresponding author}



\begin{abstract}
Node-Kayles is an impartial combinatorial game played on graphs, in which the two players alternately select a vertex and remove it together with all of its neighbours. By the Sprague--Grundy theorem, every position has a Grundy value that determines its outcome under optimal play. In this paper we study the Grundy values of Node-Kayles on level-regular trees $\T{a}$, $a=(a_1,\dots,a_d)$, in which every vertex at level $i-1$ has exactly $a_i$ children. We determine the Grundy values of all level-regular trees of depth at most two, show that $\G(\T{a})\le d+1$, and prove that $\G(\T{a})>0$ whenever $a_1a_2a_3$ is even, using an explicit mirror strategy. We prove a Leaf-Pair Lemma, which shows that the last branching parameter $a_d$ matters only through its parity, and a Twin Lemma for subgraphs that come in isomorphic pairs; as a consequence, $\G(\T{1,2k,s})=1$ for every $k\ge1$ and every finite sequence $s$. For trees containing a long run of ones, we give a general periodicity criterion: if the Grundy values of a tree with a pendant path of length $h$ are $34$-periodic on a finite window, then they are $34$-periodic for all larger $h$. This criterion extends to paths joining two fixed subtrees, and it explains the eventual periodicity of the families $\T{1^k,2}$ and $\T{2,1^k,2}$. We also show that, in this setting, eventual periodicity is equivalent to the boundedness of the Grundy values. Finally, supported by computations, we conjecture that the Grundy values are eventually periodic, with period $2$, in each branching parameter, and eventually periodic in the length of every run of ones.
\end{abstract}


\begin{highlights}
	\item Grundy values of all level-regular trees of depth at most two are determined.
	\item The Grundy value of a level-regular tree of depth $d$ is at most $d+1$.
	\item Leaf-pair Lemma and Twin Lemma simplify pendant leaves and isomorphic pairs.
	\item A finite check proves periodicity in the length of long paths in trees.
	\item Periodicity in the path length is equivalent to bounded Grundy values.
\end{highlights}


\maketitle

\section{Introduction}
Impartial combinatorial games form a central object of study in combinatorial game theory. They are two-player, perfect-information, finite games without chance, in which the set of legal moves from any position is the same for both players and the winner is determined under the normal-play convention by making the last move \citep{berlekamp2001winning,brown2020nimber, conway2000numbers}. Classic examples include Nim, Subtraction and Taking-and-Breaking games, and various graph-based games such as Node-Kayles \citep{albert2019lessons, berlekamp2001winning, brown2020nimber}. The unifying mathematical framework for such games is provided by the Sprague--Grundy theorem, developed independently by \citet{sprague1935mathematische} and \citet{grundy1939mathematics}  and popularized in the work of \cite{berlekamp2001winning} and \cite{conway2000numbers}.

The Sprague--Grundy theorem asserts that every impartial game position $G$ is equivalent, under normal play, to a single Nim heap of some nonnegative integer size, called its Nimber or Grundy value $\G(G)$. The terminal position has Grundy value $0$, and a position has Grundy value
$0$ if and only if it is a $\mathcal{P}$-position, i.e., a second-player win with perfect play. The Grundy value of a position is defined recursively by the minimum excluded value (mex) rule: if
$O(G)$ is the set of options (positions reachable in one move) from $G$ and $S = \left\{\G(H) : H \in O(G)\right\}$ is the set of Grundy values of these options, then
\begin{equation*}
	\G(G) = \mathrm{mex}(S),
\end{equation*}
where $\mathrm{mex}(S)$ is the least nonnegative integer not in $S$ \citep{brown2020nimber, bodlaender2002kayles}. This minimum excluded value (mex) rule gives a purely combinatorial procedure for computing Grundy values from the game graph.

A second pillar of Sprague--Grundy theory is the behaviour of impartial games under disjoint sum. When a game decomposes into independent components $G = G_1 + G_2 + \dots + G_k$ (each move affects exactly one component), the Grundy value of $G$ is the nim-sum (bitwise XOR) of the Grundy values of its components:
\begin{equation*}
	\G(G) = \bigoplus_{i=1}^{k}\G(G_i) = \G(G_1) \oplus \G(G_2) \oplus \dots \oplus \G(G_k).
\end{equation*}
Here $\oplus$ denotes binary addition without carry \citep{albert2019lessons,bodlaender2002kayles,brown2020nimber}. This rule reduces the analysis of complex disjoint unions to that of their connected components and makes Grundy values a powerful algebraic invariant of impartial games.

Node-Kayles is a generalisation of Kayles \citep{berlekamp2003winning}, independently introduced by  \citet{dudeney2002canterbury} and \citet{loyd1976sam}. Playing Node-Kayles on a path is equivalent to a particular \textit{Take-and-Break} game introduced by \citet{dawson1935caissa}, and now known as \textit{Dawson's chess}, which corresponds to the octal game $.137$  \citep[see][for more details]{berlekamp2003winning,conway2000numbers,flammenkamp2000sprague}. This consists of removing $1,2$ or $3$ counters with certain options to shorten or split the heap, (see Table \ref{octal}). 
\begin{table}[]
	\caption{Interpretation of Octal game code}\label{octal}
	\begin{tabular}{|c|c|l|}
		\hline
		Octal digit & \begin{tabular}[c]{@{}c@{}}Binary\\ (bits)\end{tabular} & \multicolumn{1}{c|}{Allowed actions for a removal of $i$ counters}                                                                                                      \\ \hline
		0           & 000                                                     & No legal move of size $i$ (the player cannot remove exactly $i$ counters).                                                                                              \\ \hline
		1           & 001                                                     & \begin{tabular}[c]{@{}l@{}}Take-all only. The player may remove exactly $i$ counters and \\ the heap becomes empty.\end{tabular}                                        \\ \hline
		2           & 010                                                     & \begin{tabular}[c]{@{}l@{}}Reduce only. The player may remove exactly $i$ counters and leave \\ a single smaller heap (the remainder is kept as one heap).\end{tabular} \\ \hline
		3           & 011                                                     & \begin{tabular}[c]{@{}l@{}}Take-all or Reduce. \\ Both options are permitted for a removal of $i$ counters.\end{tabular}                                                \\ \hline
		4           & 100                                                     & \begin{tabular}[c]{@{}l@{}}Split only. After removing $i$ counters the remainder must be divided \\ into two non‑empty heaps.\end{tabular}                              \\ \hline
		5           & 101                                                     & Take-all or Split.                                                                                                                                                      \\ \hline
		6           & 110                                                     & Reduce or Split.                                                                                                                                                        \\ \hline
		7           & 111                                                     & Take-all, Reduce, or Split                                                                                                                                              \\ \hline
	\end{tabular}
\end{table}
In that
sense, $.137$ means: one single vertex can only be removed if it has no remaining neighbours. Two vertices can only be removed if they form one connected component or if after their removal the remaining part of that connected component is still connected. For three vertices we have more possibilities. We are allowed to remove them if the remaining vertices are left in at most two connected components.

\begin{theorem}\label{thm:path}
	The Grundy value sequence ($\G(P_n):n \in \mathbb{N}$) is eventually periodic sequence with preperiod of length $51$ and period $34$.
\end{theorem}
The periodic part of the sequence for $n \equiv 0,1,2 \dots ,33 \pmod{34}$ is
\begin{equation*}
	8,1,1,2,0,3,1,1,0,3,3,2,2,4,4,5,5,9,3,3,0,1,1,3,0,2,1,1,0,4,5,3,7,4.
\end{equation*}
The Grundy value sequence $(\G(P_n))$ coincides with OEIS sequence \href{https://oeis.org/A002187}{A002187}.

Node-Kayles, also known as the domination game in some graph-theoretic contexts \citep{duchene2009combinatorial}, is a canonical impartial game played on a simple graph $G = (V,E)$. In the normal version of Node-Kayles, players alternate choosing a vertex $v \in V$; a move removes $v$, all its neighbours $N_{G}(v)$, and all incident edges from the graph, leaving the induced subgraph $G_{v} = G - N_{G}[v]$, where $N_{G}[v] = \{v\} \cup N_{G}(v)$ is the closed neighbourhood of $v$. The player who first has no legal move loses. The game is impartial because both players have the same set of moves from any given position, and it is finite under normal play. Node-Kayles has been studied from both complexity-theoretic and structural perspectives: deciding the winner on general graphs is PSPACE-complete \citep{bodlaender2015exact,fleischer2004kayles}, while polynomial-time algorithms exist for various restricted graph classes such as graphs of bounded asteroidal number, cocomparability graphs, cographs, and others \citep{bodlaender2002kayles,guignard2009compound}.

From the Sprague--Grundy viewpoint, each simple graph $G$ has an associated Grundy value $\G(G)$, defined recursively by
\begin{equation*}
	\G(G) = \mathrm{mex}\{\G(G_{v}): v \in V(G)\},
\end{equation*}
where $\G(\emptyset) = 0$. Moreover, if $G = H \cup K$ is a disjoint union of graphs, meaning that $V(G)=V(H)\cup V(K)$ and $E(G)=E(H)\cup E(K)$ with $V(H)\cap V(K)=\emptyset$, then the Grundy value satisfies
\begin{equation*}
	\G(G) = \G(H) \oplus \G(K).
\end{equation*}

To illustrate these definitions, consider Node-Kayles on some simple graphs and compute their respective Grundy values. Let $K_n$ denote a complete graph with $n$ vertices. Since every vertex of $K_n$ has closed neighbourhood $V(K_n)$, every move leaves the empty graph. Thus
\begin{equation*}
	\G(K_n) = \mathrm{mex}\{\G(\emptyset)\} = \mathrm{mex}\{0\} = 1.
\end{equation*}
for every integer $n \ge 1$. Let $P_n$ denote a path on $n$ vertices. For $P_1$, the only vertex $v$ has closed neighbourhood $N[v] = V(P_1)$, so the only move also leaves the empty graph. Thus
\begin{equation*}
	\G(P_1) = \mathrm{mex}\{\G(\emptyset)\} = \mathrm{mex}\{0\} = 1.
\end{equation*}
For $P_2$, either endpoint has closed neighbourhood equal to all of $P_2$, so any move again leaves the empty graph, and
\begin{equation*}
	\G(P_2) = \mathrm{mex}\{\G(\emptyset)\} = 1.
\end{equation*}
For $P_3$ with vertices $v_1,v_2$ and $v_3$, playing at $v_2$ removes all vertices and leaves the empty graph (Grundy value of $0$), while playing at $v_1$ or $v_3$ leaves a single isolated vertex $P_1$ (Grundy value of $1$). Hence the set of option Grundy values from $P_3$ is $\{0,1\}$, and
\begin{equation*}
	\G(P_3) = \mathrm{mex}\{0,1\} = 2.
\end{equation*}
If $G$ is the graph in Figure \ref{fig:ex1}, then as seen in Figure \ref{fig:ex2}, the induced subgraphs $G_{v_1}$ and $G_{v_2}$ are isomorphic to $P_1 \cup K_3$, while the induced subgraphs $G_{v_4}$ and $G_{v_5}$ are isomorphic to $P_1 \cup P_1$ and $G_{v_3}$ and $G_{v_6}$ are isomorphic to $P_1$ and $P_3$, respectively.
\begin{figure}
	\centering
	\begin{tikzpicture}[scale=0.75]
		\tikzset{enclosed/.style={draw, circle, inner sep=0pt, minimum size=.15cm, fill=black}}
		
		\node[enclosed, label={left, yshift=.2cm: $v_1$}] (1) at (0.75,3.25) {};
		\node[enclosed, label={above, yshift=0cm: $v_3$}] (3) at (2.5,2) {};
		\node[enclosed, label={left, yshift=-.2cm: $v_2$}] (2) at (0.75,0.75) {};
		\node[enclosed, label={left, yshift=.2cm: $v_4$}] (4) at (4.25,3.25) {};
		\node[enclosed, label={above, yshift=0cm: $v_6$}] (6) at (6,2) {};
		\node[enclosed, label={left, yshift=-.2cm: $v_5$}] (5) at (4.25,0.75) {};
		
		\draw (1) -- (3) node[midway, sloped, above] (edge1) {};
		\draw (2) -- (3) node[midway, right] (edge2) {};
		\draw (3) -- (4) node[midway, sloped, above] (edge3) {};
		\draw (3) -- (5) node[midway, sloped, above] (edge4) {};
		\draw (5) -- (4) node[midway, sloped, above] (edge5) {};
		\draw (6) -- (4) node[midway, sloped, above] (edge6) {};
		\draw (6) -- (5) node[midway, sloped, above] (edge7) {};
	\end{tikzpicture}
	\caption{A simple graph $G$ \label{fig:ex1}}
\end{figure}
\begin{figure}
	\centering
	\begin{tikzpicture}[scale=0.75]
		\tikzset{enclosed/.style={draw, circle, inner sep=0pt, minimum size=.15cm, fill=black}}
		\node[] (1) at (1.25,0) {$G_{v_1}$};
		\node[] (2) at (4,0) {$G_{v_2}$};
		\node[] (3) at (6.75,0) {$G_{v_3}$};
		\node[] (4) at (9.5,0) {$G_{v_4}$};
		\node[] (5) at (12.25,0) {$G_{v_5}$};
		\node[] (6) at (15,0) {$G_{v_6}$};
		
		\node[enclosed, label={left, xshift=.1cm ,yshift=.2cm: $v_2$}] (12) at (0.5,1.5) {};
		\node[enclosed, label={above, yshift=.0cm: $v_4$}] (14) at (1,2.25) {};
		\node[enclosed, label={left, xshift=.1cm, yshift=.0cm: $v_5$}] (15) at (1,0.75) {};
		\node[enclosed, label={above, yshift=.0cm: $v_6$}] (16) at (2,1.5) {};
		\draw (14) -- (15);
		\draw (14) -- (16);
		\draw (16) -- (15);
		
		\node[enclosed, label={left, xshift=.1cm, yshift=.2cm: $v_1$}] (21) at (3.25,1.5) {};
		\node[enclosed, label={above, yshift=.0cm: $v_4$}] (24) at (3.75,2.25) {};
		\node[enclosed, label={left, xshift=.1cm, yshift=.0cm: $v_5$}] (25) at (3.75,0.75) {};
		\node[enclosed, label={above, yshift=.0cm: $v_6$}] (26) at (4.75,1.5) {};
		\draw (24) -- (25);
		\draw (24) -- (26);
		\draw (26) -- (25);
		
		\node[enclosed, label={left, yshift=.0cm: $v_6$}] (36) at (6.75,1.5) {};
		
		\node[enclosed, label={above, yshift=.0cm: $v_1$}] (41) at (9.5,2.25) {};
		\node[enclosed, label={above, yshift=.0cm: $v_2$}] (42) at (9.5,0.75) {};
		
		\node[enclosed, label={above, yshift=.0cm: $v_1$}] (51) at (12.25,2.25) {};
		\node[enclosed, label={above, yshift=.0cm: $v_2$}] (52) at (12.25,0.75) {};
		
		\node[enclosed, label={above, yshift=.0cm: $v_1$}] (61) at (14.5,2.25) {};
		\node[enclosed, label={left, yshift=.0cm: $v_2$}] (62) at (14.5,0.75) {};
		\node[enclosed, label={above, yshift=.0cm: $v_3$}] (63) at (15.5,1.5) {};
		\draw (61) -- (63);
		\draw (62) -- (63);
	\end{tikzpicture}
	\caption{The induced subgraphs $G_{v_1}, G_{v_2}, G_{v_3}, G_{v_4}, G_{v_5}$ and $G_{v_6}$ \label{fig:ex2}}
\end{figure}

\noindent Therefore,
\begin{align*}
	\G(G) &= \mathrm{mex}\{\G(G_{v_i}): i = 1,2,3,4,5,6\}\\
	&= \mathrm{mex}\{\G(P_1 \cup K_3), \G(P_1 \cup P_1), \G(P_1), \G(P_3)\}\\
	&= \mathrm{mex}\{\G(P_1) \oplus \G(K_3), \G(P_1) \oplus \G(P_1), \G(P_1), \G(P_3)\}\\
	&= \mathrm{mex}\{1 \oplus 1, 1 \oplus 1, 1, 2\}\\
	&= \mathrm{mex}\{0, 1, 2\} = 3.
\end{align*}
This simple example shows how the $\mathrm{mex}$ rule and disjoint-union rule ($\oplus$) combine to yield Grundy values in practice.

This analysis has been extended to a wide range of Node-Kayles graph families, including powers and variants of paths, lattice graphs, prism graphs, chained and linked cliques, linked cycles and diamonds, hypercubes, and generalised Petersen graphs, often yielding explicit or recursive descriptions of their Grundy value sequences \citep{brown2020nimber, fleischer2004kayles, guignard2009compound}.

In this work, we investigate the Grundy values of Node-Kayles on level-regular trees. In Section~\ref{sec2}, we determine the Grundy values of the infinite families $\left(\T{n}: n \in \mathbb{N}\right)$ and $\left(\T{n,m}: n,m \in \mathbb{N}\right)$, and explain why a direct recursive computation becomes impractical from depth three on. We then prove general properties of level-regular trees $\T{a}$ with $a=(a_1,\dots,a_d)$: the Grundy value is at most $d+1$, and it is positive whenever $a_1a_2a_3$ is even, which we show with an explicit mirror strategy. We also prove a Leaf-Pair Lemma, which shows that the last branching parameter $a_d$ affects the Grundy value only through its parity, and a Twin Lemma for rooted subgraphs that come in isomorphic pairs; as a consequence, $\G\left(\T{1,2k,s}\right)=1$ for every $k\ge1$ and every finite sequence $s$.

In Section~\ref{sec3}, we study trees that contain a long run of ones, that is, a long path attached to a fixed rooted tree. We prove a periodicity criterion: if the Grundy values are periodic with period $34$ on a finite window of path lengths, then they are periodic with period $34$ for all larger lengths. This reduces each family, such as $\left(\T{1^k,2}\right)$, to a finite computation. We show that, in this setting, eventual periodicity is equivalent to the boundedness of the Grundy values, and we extend the criterion to a path joining two fixed rooted trees, which covers families such as $\left(\T{2,1^k,2}\right)$. We conclude with conjectures on the eventual periodicity of the Grundy values in each branching parameter and in the length of every run of ones. These results provide new infinite families of graphs whose Grundy values can be determined rigorously, and they contribute techniques for studying periodicity in octal and graph-based impartial games.
\section{Node-Kayles on level-regular trees}\label{sec2}
In this section, we investigate the Grundy value associated with Node-Kayles played on level-regular trees.
\begin{definition}
	For a sequence $a = (a_1, a_2, \dots, a_d)$ of positive integers, the level-regular tree $\T{a}$ is a rooted tree of depth $d$ (with the root at level $0$), where every vertex at level $i-1$ has exactly $a_{i}$ children, for $i = 1,\dots, d$; see Figure \ref{fig:level-regular-tree}. For $0 \le i \le d$, let
	\begin{equation*}
		F_i := \T{a_{i+1},\dots, a_d}
	\end{equation*}
	be the full subtree hanging from a vertex at level $i$. Thus $F_0 = \T{a}$ and $F_d = P_1$.
	We write $1^h$ for a string of $h$ ones, so $\T{1^h,s}$ is $\T{s}$ with a path of $h$ vertices attached to its root.
\end{definition}
\begin{figure}
	\centering
	\begin{tikzpicture}[scale=1]
		\fill[blue!7,rounded corners=6pt] (0.1,-0.9) rectangle (3.25,-3.9);
		\node[blue!50!black,font=\small] at (1.65,-4.2) {$F_1=\T{3,2}$};
		\fill[red!8,rounded corners=5pt] (-3.3,-2.12) rectangle (-2.2,-3.9);
		\node[red!50!black,font=\small] at (-2.75,-4.2) {$F_2=\T{2}$};
		\foreach \k in {0,...,11} \coordinate (l\k) at ({(\k-5.5)*0.55},-3.6);
		\foreach \m in {0,...,5} \coordinate (b\m) at ({(2*\m-5)*0.55},-2.4);
		\foreach \i in {0,1} \coordinate (a\i) at ({(6*\i-3)*0.55},-1.2);
		\coordinate (r) at (0,0);
		\foreach \i in {0,1} \draw (r) -- (a\i);
		\foreach \m in {0,...,5} {\pgfmathtruncatemacro{\i}{int(\m/3)} \draw (a\i) -- (b\m);}
		\foreach \k in {0,...,11} {\pgfmathtruncatemacro{\m}{int(\k/2)} \draw (b\m) -- (l\k);}
		\foreach \p in {r,a0,a1,b0,b1,b2,b3,b4,b5,l0,l1,l2,l3,l4,l5,l6,l7,l8,l9,l10,l11} \fill (\p) circle (1.6pt);
		\foreach \y/\t in {0/0,-1.2/1,-2.4/2,-3.6/3} \node[font=\small,anchor=east] at (-4.0,\y) {level \t};
		\foreach \y/\t in {-0.6/{a_1=2},-1.8/{a_2=3},-3.0/{a_3=2}} \node[font=\small,anchor=west] at (3.6,\y) {$\t$};
	\end{tikzpicture}
	\caption{The level-regular tree $\T{2,3,2}$. Every vertex at level $i-1$ has $a_i$ children. The shaded subtrees are copies of $F_1$ (below a level-$1$ vertex) and $F_2$ (below a level-$2$ vertex).}\label{fig:level-regular-tree}
\end{figure}
For a statement $S$, let
\begin{equation*}
	[S]:= \begin{cases}
		1 \quad \text{if the statement $S$ is true};\\
		0 \quad \text{if the statement $S$ is false.}\\
	\end{cases}
\end{equation*}
be the indicator. We write $n \cdot g$ for the nim-sum of $n$ copies of $g$, so that $n \cdot g = [n \text{ is odd}]g$.
\begin{theorem}\label{thm1}
	The Grundy value sequence $\left(\G\left(\T{n}\right) : n \in \mathbb{N}\right)$ is given by the following:
	\begin{equation*}
		\G\left(\T{n}\right) = \begin{cases}
			1,	\quad \text{if $n$ is odd};\\
			2,	\quad \text{if $n$ is even}
		\end{cases}
	\end{equation*}
	for all $n \in \mathbb{N}$.
\end{theorem}
\begin{proof}
	By definition,
	\begin{equation*}
		\G\left(\T{n}\right) = \mathrm{mex}\left\{\G\left(\emptyset\right), \bigoplus_{i=1}^{n-1}\G\left(P_1\right)\right\}
	\end{equation*}
	with $\G\left(\emptyset\right) = 0$ and $\G\left(P_1\right) = 1$. By the definition of disjoint union,
	\begin{equation*}
		\bigoplus_{i=1}^{n-1}\G\left(P_1\right) = \begin{cases}
			0,	\quad \text{if $n$ is odd};\\
			1,	\quad \text{if $n$ is even}.
		\end{cases}
	\end{equation*}
	Therefore, 
	\begin{equation*}
		\G\left(\T{n}\right) = \begin{cases}
			1,	\quad \text{if $n$ is odd};\\
			2,	\quad \text{if $n$ is even.}
		\end{cases}
	\end{equation*}
\end{proof}
\begin{theorem}\label{thm2}
	The Grundy value sequence $\left( \G\left(\T{n,m}\right) : n,m \in \mathbb{N} \right)$ is given by the following:
	\begin{equation*}
		\G\left(\T{n,m}\right) = \begin{cases}
			1,	\quad \text{if $m$ is even};\\
			2,	\quad \text{if $n,m$ are odd};\\
			3,	\quad \text{if $n$ is even and $m$ is odd}
		\end{cases}
	\end{equation*}
	for all $n,m \in \mathbb{N}$.
\end{theorem}
\begin{proof}
	The first player has three types of moves.
	\begin{itemize}
		\item play at the root: $\bigoplus_{i=1}^{nm}\G\left(P_1\right)$
		\item play at level $1$: $\bigoplus_{i=1}^{n-1}\G\left(\T{m}\right)$
		\item play at level $2$: $\G\left(\T{n-1,m}\right) \oplus \bigoplus_{i=1}^{m-1}\G\left(P_1\right)$
	\end{itemize}
	By definition,
	\begin{equation*}
		\G\left(\T{n,m}\right) = \mathrm{mex}\left\{\bigoplus_{i=1}^{n m}\G\left(P_1\right), \bigoplus_{i=1}^{n-1}\G\left(\T{m}\right), \G\left(\T{n-1,m}\right) \oplus \bigoplus_{i=1}^{m-1}\G\left(P_1\right)\right\}
	\end{equation*}
	Suppose $m$ is even. It is easy to see that $\G\left(\T{1,m}\right) = \G\left(\T{m+1}\right) =  1$, $\bigoplus_{i=1}^{nm}\G\left(P_1\right) = 0$ and
	\begin{equation*}
		\bigoplus_{i=1}^{n-1}\G\left(\T{m}\right) = \begin{cases}
			0,	\quad \text{if $n$ is odd};\\
			2,	\quad \text{if $n$ is even}.
		\end{cases}
	\end{equation*} Using induction, assume that $\G\left(\T{k,m}\right) = 1$. Then
	\begin{align*}
		\G\left(\T{k+1,m}\right) &= \mathrm{mex}\left\{\bigoplus_{i=1}^{(k+1)m}\G\left(P_1\right), \bigoplus_{i=1}^{k}\G\left(\T{m}\right), \G\left(\T{k,m}\right) \oplus \bigoplus_{i=1}^{m-1}\G\left(P_1\right)\right\}\\
		&= \mathrm{mex}\left\{0, (0 \text{ or } 2), \G\left(\T{k,m}\right) \oplus  1\right\}\\
		&= \mathrm{mex}\left\{0, (0 \text{ or } 2), 1 \oplus  1\right\}\\
		&= \mathrm{mex}\left\{0, (0 \text{ or } 2)\right\}\\
		&= 1.
	\end{align*}
	Therefore, $\G\left(\T{n,m}\right) = 1$, completing the inductive step.
	If $m$ is odd, then $\G\left(\T{1,m}\right) = \G\left(\T{m+1}\right) =  2$,
	\begin{equation*}
		\bigoplus_{i=1}^{nm}\G\left(P_1\right) = \begin{cases}
			0,	\quad \text{if $n$ is even};\\
			1,	\quad \text{if $n$ is odd}
		\end{cases}
	\end{equation*}
	and
	\begin{equation*}
		\bigoplus_{i=1}^{n-1}\G\left(\T{m}\right) = \begin{cases}
			0,	\quad \text{if $n$ is odd};\\
			1,	\quad \text{if $n$ is even}.
		\end{cases}
	\end{equation*}
	It remains to show that $\G\left(\T{2,m}\right) =  3$. We have
	\begin{align*}
		\G\left(\T{2,m}\right) &= \mathrm{mex}\left\{\bigoplus_{i=1}^{2 m}\G\left(P_1\right), \G\left(\T{m}\right), \G\left(\T{1,m}\right) \oplus \bigoplus_{i=1}^{m-1}\G\left(P_1\right)\right\}\\
		&= \mathrm{mex}\left\{0, 1, 2 \oplus 0\right\}\\
		&= \mathrm{mex}\left\{0, 1, 2 \right\}\\
		&= 3
	\end{align*}
	We are ready to finish our proof by induction. Assume that for some integer $k \ge 2$, 
	\begin{equation*}
		\G\left(\T{k,m}\right) = \begin{cases}
			2,	\quad \text{if $k$ is odd };\\
			3,	\quad \text{if $k$ is even}.
		\end{cases}
	\end{equation*}
	If $k+1$ is odd, we get
	\begin{align*}
		\G\left(\T{k+1,m}\right) &= \mathrm{mex}\left\{\bigoplus_{i=1}^{(k+1) m}\G\left(P_1\right), \bigoplus_{i=1}^{k}\G\left(\T{m}\right), \G\left(\T{k,m}\right) \oplus \bigoplus_{i=1}^{m-1}\G\left(P_1\right)\right\}\\
		&= \mathrm{mex}\left\{1,0,3 \oplus 0\right\}\\
		&= \mathrm{mex}\left\{0,1,3\right\}\\
		&= 2
	\end{align*}
	Similarly, if $k+1$ is even,
	\begin{align*}
		\G\left(\T{k+1,m}\right) &= \mathrm{mex}\left\{\bigoplus_{i=1}^{(k+1) m}\G\left(P_1\right), \bigoplus_{i=1}^{k}\G\left(\T{m}\right), \G\left(\T{k,m}\right) \oplus \bigoplus_{i=1}^{m-1}\G\left(P_1\right)\right\}\\
		&= \mathrm{mex}\left\{0,1,2 \oplus 0\right\}\\
		&= \mathrm{mex}\left\{0,1,2\right\}\\
		&= 3
	\end{align*}
\end{proof}

The technique in the proof of Theorem \ref{thm2} is to recursively compute the Grundy values of induced subgraphs. This works because every option is a disjoint union of trees from the same family whose Grundy values are already known. Unfortunately, many induced subgraphs of level-regular trees $\T{a}$ are not level-regular when $d>2$. As a result, this method is far too tedious to be applied to larger graphs. For example, to determine the Grundy value of the Node-Kayles game on $\T{2,2,3}$ we need to first find the Grundy value of the following graph
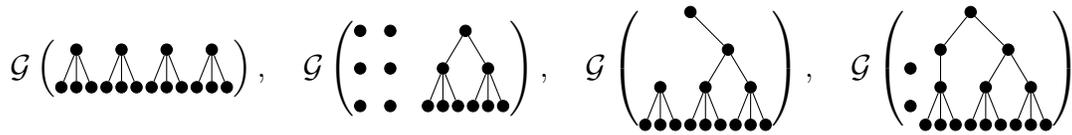
\begin{equation*}
	\G\left(\vcenter{\hbox{\begin{tikzpicture}
				\tikzset{enclosed/.style={draw, circle, inner sep=0pt, minimum size=.15cm, fill=black}}
				\node[enclosed] (4) at (0.2,0.5) {};
				\node[enclosed] (5) at (0.8,0.5) {};
				\node[enclosed] (6) at (1.4,0.5) {};
				\node[enclosed] (8) at (2,0.5) {};				
				\node[enclosed] (42) at (0.8,0) {};
				\node[enclosed] (22) at (0.2,0) {};
				\node[enclosed] (21) at (0,0) {};
				\node[enclosed] (23) at (0.4,0) {};
				\node[enclosed] (41) at (0.6,0) {};
				\node[enclosed] (43) at (1,0) {};
				\node[enclosed] (62) at (1.4,0) {};
				\node[enclosed] (82) at (2,0) {};
				\node[enclosed] (81) at (1.8,0) {};
				\node[enclosed] (83) at (2.2,0) {};
				\node[enclosed] (61) at (1.2,0) {};
				\node[enclosed] (63) at (1.6,0) {};
				
				\draw (4) -- (22);
				\draw (4) -- (21);
				\draw (4) -- (23);
				\draw (5) -- (41);
				\draw (5) -- (42);
				\draw (5) -- (43);
				\draw (6) -- (62);
				\draw (6) -- (61);
				\draw (6) -- (63);
				\draw (8) -- (81);
				\draw (8) -- (82);
				\draw (8) -- (83);
	\end{tikzpicture}}}
	\right), 
	\quad \G\left(\vcenter{\hbox{\begin{tikzpicture}
				\tikzset{enclosed/.style={draw, circle, inner sep=0pt, minimum size=.15cm, fill=black}}
				\node[enclosed] (3) at (1.7,1) {};
				\node[enclosed] (6) at (1.4,0.5) {};
				\node[enclosed] (8) at (2,0.5) {};
				
				\node[enclosed] (42) at (0.7,0.5) {};
				\node[enclosed] (22) at (0.3,0.5) {};
				\node[enclosed] (21) at (0.7,1) {};
				\node[enclosed] (23) at (0.3,1) {};
				\node[enclosed] (41) at (0.3,0) {};
				\node[enclosed] (43) at (0.7,0) {};
				\node[enclosed] (62) at (1.4,0) {};
				\node[enclosed] (82) at (2,0) {};
				\node[enclosed] (81) at (1.8,0) {};
				\node[enclosed] (83) at (2.2,0) {};
				\node[enclosed] (61) at (1.2,0) {};
				\node[enclosed] (63) at (1.6,0) {};
				
				\draw (3) -- (6);
				\draw (3) -- (8);
				\draw (6) -- (62);
				\draw (6) -- (61);
				\draw (6) -- (63);
				\draw (8) -- (81);
				\draw (8) -- (82);
				\draw (8) -- (83);
	\end{tikzpicture}}}
	\right), \quad \G\left(\vcenter{\hbox{\begin{tikzpicture}
				\tikzset{enclosed/.style={draw, circle, inner sep=0pt, minimum size=.15cm, fill=black}}
				\node[enclosed] (1) at (1.2,1.5) {};
				\node[enclosed] (3) at (1.7,1) {};
				\node[enclosed] (5) at (0.8,0.5) {};
				\node[enclosed] (6) at (1.4,0.5) {};
				\node[enclosed] (8) at (2,0.5) {};
				
				\node[enclosed] (42) at (0.8,0) {};
				\node[enclosed] (41) at (0.6,0) {};
				\node[enclosed] (43) at (1,0) {};
				\node[enclosed] (62) at (1.4,0) {};
				\node[enclosed] (82) at (2,0) {};
				\node[enclosed] (81) at (1.8,0) {};
				\node[enclosed] (83) at (2.2,0) {};
				\node[enclosed] (61) at (1.2,0) {};
				\node[enclosed] (63) at (1.6,0) {};
				
				\draw (1) -- (3);
				\draw (3) -- (6);
				\draw (3) -- (8);
				\draw (5) -- (41);
				\draw (5) -- (42);
				\draw (5) -- (43);
				\draw (6) -- (62);
				\draw (6) -- (61);
				\draw (6) -- (63);
				\draw (8) -- (81);
				\draw (8) -- (82);
				\draw (8) -- (83);
	\end{tikzpicture}}}
	\right), \quad \G\left(\vcenter{\hbox{\begin{tikzpicture}
				\tikzset{enclosed/.style={draw, circle, inner sep=0pt, minimum size=.15cm, fill=black}}
				\node[enclosed] (1) at (1.2,1.5) {};
				\node[enclosed] (2) at (0.8,1) {};
				\node[enclosed] (3) at (1.7,1) {};
				\node[enclosed] (5) at (0.8,0.5) {};
				\node[enclosed] (6) at (1.4,0.5) {};
				\node[enclosed] (8) at (2,0.5) {};
				
				\node[enclosed] (9) at (0.4,0.25) {};
				\node[enclosed] (0) at (0.4,0.75) {};
				
				\node[enclosed] (42) at (0.8,0) {};
				\node[enclosed] (41) at (0.6,0) {};
				\node[enclosed] (43) at (1,0) {};
				\node[enclosed] (62) at (1.4,0) {};
				\node[enclosed] (82) at (2,0) {};
				\node[enclosed] (81) at (1.8,0) {};
				\node[enclosed] (83) at (2.2,0) {};
				\node[enclosed] (61) at (1.2,0) {};
				\node[enclosed] (63) at (1.6,0) {};
				
				\draw (1) -- (2);
				\draw (2) -- (5);
				\draw (1) -- (3);
				\draw (3) -- (6);
				\draw (3) -- (8);
				\draw (5) -- (41);
				\draw (5) -- (42);
				\draw (5) -- (43);
				\draw (6) -- (62);
				\draw (6) -- (61);
				\draw (6) -- (63);
				\draw (8) -- (81);
				\draw (8) -- (82);
				\draw (8) -- (83);
	\end{tikzpicture}}}
	\right)
\end{equation*}
before we find the Grundy value of $\T{2,2,3}$.
\begin{figure}
	\centering
	\begin{tikzpicture}[scale=0.75]
		\tikzset{enclosed/.style={draw, circle, inner sep=0pt, minimum size=.15cm, fill=black}}
		
		\node[enclosed] (1) at (5,4) {};
		\node[enclosed] (2) at (3,3) {};
		\node[enclosed] (3) at (7,3) {};
		\node[enclosed] (4) at (2,2) {};
		\node[enclosed] (5) at (4,2) {};
		\node[enclosed] (6) at (6,2) {};
		\node[enclosed] (8) at (8,2) {};
		\node[enclosed] (42) at (4,1) {};
		\node[enclosed] (22) at (2,1) {};
		\node[enclosed] (21) at (1.5,1) {};
		\node[enclosed] (23) at (2.5,1) {};
		\node[enclosed] (41) at (3.5,1) {};
		\node[enclosed] (43) at (4.5,1) {};
		\node[enclosed] (62) at (6,1) {};
		\node[enclosed] (82) at (8,1) {};
		\node[enclosed] (81) at (7.5,1) {};
		\node[enclosed] (83) at (8.5,1) {};
		\node[enclosed] (61) at (5.5,1) {};
		\node[enclosed] (63) at (6.5,1) {};
		
		\draw (1) -- (2) node[midway, sloped, above] (edge1) {};
		\draw (1) -- (3) node[midway, right] (edge2) {};
		\draw (2) -- (4) node[midway, sloped, above] (edge3) {};
		\draw (2) -- (5) node[midway, sloped, above] (edge4) {};
		\draw (3) -- (6) node[midway, sloped, above] (edge5) {};
		\draw (3) -- (8) node[midway, sloped, above] (edge6) {};
		\draw (4) -- (22) node[midway, sloped, above] (edge7) {};
		\draw (4) -- (21) node[midway, sloped, above] (edge7) {};
		\draw (4) -- (23) node[midway, sloped, above] (edge7) {};
		\draw (5) -- (41) node[midway, sloped, above] (edge7) {};
		\draw (5) -- (42) node[midway, sloped, above] (edge7) {};
		\draw (5) -- (43) node[midway, sloped, above] (edge7) {};
		\draw (6) -- (62) node[midway, sloped, above] (edge7) {};
		\draw (6) -- (61) node[midway, sloped, above] (edge7) {};
		\draw (6) -- (63) node[midway, sloped, above] (edge7) {};
		\draw (8) -- (81) node[midway, sloped, above] (edge7) {};
		\draw (8) -- (82) node[midway, sloped, above] (edge7) {};
		\draw (8) -- (83) node[midway, sloped, above] (edge7) {};
	\end{tikzpicture}
	\caption{The level-regular tree $\T{2,2,3}$ \label{fig:ex4}}
\end{figure}

The complexity of this process drastically increases when we try to determine $\G\left(\T{a}\right)$. Hence, in order to study the Grundy values of Node-Kayles on level-regular trees in general, we need to adopt another approach.

Despite this complexity, we can still prove some general properties of these Grundy values.
\begin{theorem}\label{thm:bound}
	For every sequence $a = (a_1, a_2, \dots, a_d)$ of positive integers, $0 \le \G(\T{a}) \le d+1$.
\end{theorem}
\begin{proof}
	By the symmetry of $\T{a}$ (every vertex at level $i-1$ has exactly $a_i$ children), moves at any two vertices of the same level lead to isomorphic subgames. Hence all moves at the same level give the same Grundy value. Since $\T{a}$ has $d+1$ levels, its options take at most $d+1$ distinct Grundy values. The mex of a set is at most its number of elements, so $0 \le \G(\T{a}) \le d+1$.
\end{proof}
\begin{theorem}\label{thm:positive}
	 Let $a = (a_1, a_2, \dots, a_d)$ with $d \ge 3$, if $a_1 a_2  a_3$ is even, then $\G(\T{a}) > 0$.
\end{theorem}
\begin{proof}
	Playing at the root deletes the root and all the vertices at level $1$. What remains are the $a_1 a_2$ subtree rooted at level $2$, all isomorphic to $F_2$. If $a_1 a_2$ is even these component pair up, so the position has Grundy value of $0$ and the mirror strategy wins.
	
	Playing at a vertex $x$ at level $1$ deletes the root, $x$ and the $a_2$ children of $x$. What remains are the $a_1 -1 $ siblings of $x$ with their subtrees (copies of $F_1$), and the $a_2 a_3$ grandchildren of $x$ with their subtrees (copies of $F_3$). If $a_1$ is odd, $a_1 -1$ is even. Moreover if $a_2 a_3$ is even, all components pair up again so the Grundy value is $0$.
	
	If $a_1 a_2 a_3$ is even, then either $a_1 a_2$ is even (play at the root) or $a_1, a_2$ are odd and $a_3$ is even (play a vertex at level $1$). In both cases the resulting position has Grundy value $0$, so $\G(\T{a}) > 0$. 
\end{proof}
The next lemma shows that the last branching parameter $a_d$ affects the Grundy value only through its parity.
\begin{lemma}[Leaf-Pair Lemma]\label{lem:leaf-pair}
	Let $G$ be a graph, $u\in V(G)$, and let $G\langle k\rangle$ be $G$ with $k$ new pendant leaves attached to $u$. Then, for every $k\ge1$,
	\begin{equation*}
		\G(G\langle k+2\rangle)=\G(G\langle k\rangle).
	\end{equation*}
\end{lemma}
\begin{proof}
	We use induction on $|V(G)|$, for all $u$ and all $k \ge 1$ simultaneously. We must show that the set of Grundy values of the available options for $G\langle k+2\rangle$ is exactly the same as the set of Grundy values of the options of $G\langle k\rangle$.
	
	We categorize the possible moves from $G\langle k+2\rangle$ into four types:
	\begin{enumerate}
		\item Playing at vertex $u$\\
		 This Removes $u$ and all the leaves in both games. It leaves $G - N_G[u]$.
		 \item Playing at a leaf\\
		 This removes that leaf and vertex $u$. It leaves $(G-u) \cup (k+1)P_1$ or $(G-u) \cup (k-1)P_1$. Their Grundy values $\G\left(G-u\right) \oplus [k+1 \text{ is odd}]$ and $\G\left(G-u\right) \oplus [k-1 \text{ is odd}]$ are equal, and this move exists in $G\langle k\rangle$ because $k\ge1$.
		 \item Playing at vertex $x \in V(G)$ with $u \in N_G(x)$\\
		 This Removes $x$ leaves $(G-N_G[x]) \cup (k+2)P_1$ or $(G-N_G[x]) \cup kP_1$, which have equal Grundy values.
		 \item Playing at vertex $x \in V(G)$ with $u \notin N_G[x]$\\
		 This Removes $x$ leaves $G'\langle k+2\rangle$ or $G'\langle k\rangle$ with $G' = G-N_G[x]$ and $u \in V(G')$. Since $|V(G')| < |V(G)|$, the induction hypothesis gives equal Grundy values.
	\end{enumerate}
	Hence $G\langle k+2\rangle$ and $G\langle k\rangle$ have the same set of option values, and therefore the same mex. Thus, $\G(G\langle k+2\rangle) = \G(G\langle k\rangle)$.
\end{proof}
\begin{figure}
	\centering
	\begin{tikzpicture}[scale=0.9,every node/.style={font=\small}]
		\draw[fill=gray!10] (0,0) ellipse (1.5 and 0.7); \node at (-0.5,0.15) {$G$};
		\fill (0,-0.7) circle (2pt); \node[above] at (0,-0.7) {$u$};
		\foreach \a in {-0.9,-0.3,0.3,0.9} {\draw (0,-0.7) -- (\a,-1.6); \fill (\a,-1.6) circle (1.8pt);}
		\node at (0,-2.15) {$G\langle k\rangle$};
		\node[font=\Large] at (2.6,-0.6) {$\equiv$};
		\begin{scope}[xshift=5.2cm]
			\draw[fill=gray!10] (0,0) ellipse (1.5 and 0.7); \node at (-0.5,0.15) {$G$};
			\fill (0,-0.7) circle (2pt); \node[above] at (0,-0.7) {$u$};
			\foreach \a in {-0.9,-0.3,0.3,0.9} {\draw (0,-0.7) -- (\a,-1.6); \fill (\a,-1.6) circle (1.8pt);}
			\foreach \a in {-1.5,1.5} {\draw[red] (0,-0.7) -- (\a,-1.6); \fill[red] (\a,-1.6) circle (1.8pt);}
			\node at (0,-2.15) {$G\langle k+2\rangle$};
		\end{scope}
	\end{tikzpicture}
	\caption{The Leaf-Pair Lemma: if $u$ already has at least one pendant leaf, two more pendant leaves (red) do not change the Grundy value.}\label{fig:leafpair}
\end{figure}
We can now cleanly apply the Leaf-Pair Lemma to the level-regular trees.
\begin{corollary}\label{cor:pair}
	Let $\T{a_1,\dots,a_{d-1},a_d}$ be a level-regular tree of depth $d$. Then $\G\left(\T{a_1,\dots,a_{d-1},a_d}\right)=\G\left(\T{a_1,\dots,a_{d-1},2-(a_d \bmod 2)}\right)$ 
\end{corollary}
\begin{proof}
	Apply Lemma \ref{lem:leaf-pair} at each vertex at level $d-1$, removing pairs of leaves until one or two remain.
\end{proof}
For a graph $X$ with a vertex $z$, call a sequence of moves in $X$ \textit{z-quiet} if no move removes $z$ (that is, $z \notin N[x]$ for each played vertex $x$). Let $\mathcal{R}(X,z)$ be the set of graphs reachable from $X$ by $z$-quiet sequences, including $X$ itself.
\begin{lemma}[Twin Lemma]\label{lem:twin}
	Let $X$ be a graph and let $z$ be a vertex of $X$ such that
	\begin{equation}\tag{C}
		\G(X'-z) = \G(X') \quad \text{for every } X' \in \mathcal{R}(X,z). 
	\end{equation}
	Let $X^+$ be obtained from $X$ by attaching at $z$ any number of rooted graphs that come in isomorphic pairs $H_1, H'_1, \dots, H_k, H'_k$ (the root of each graph joined to $z$). Then $\G(X^+) = \G(X)$.
\end{lemma}
\begin{figure}
	\centering
	\begin{tikzpicture}[scale=0.9,every node/.style={font=\small}]
		\draw[fill=gray!10] (0,0) ellipse (1.8 and 0.8);
		\node at (-0.6,0.25) {$X$};
		\node[font=\normalsize] at (-2.4,0.7) {$X^{+}$};
		\foreach \a/\c/\l in {-2.0/blue/{H_1},-1.0/blue/{H_1'},1.0/orange/{H_2},2.0/orange/{H_2'}} {
			\draw (0,-0.8) -- (\a,-1.8);
			\draw[fill=\c!20] (\a,-1.8) -- (\a-0.38,-2.75) -- (\a+0.38,-2.75) -- cycle;
			\node[font=\scriptsize] at (\a,-2.45) {$\l$};
			\fill (\a,-1.8) circle (1.8pt);}
		\fill (0,-0.8) circle (2.2pt); \node[above] at (0,-0.78) {$z$};
		\draw[<->,red,thick] (-1.85,-2.95) to[bend right=50] node[below,font=\scriptsize] {(1),(3)} (-1.15,-2.95);
		\draw[<->,red,thick] (1.15,-2.95) to[bend right=50] node[below,font=\scriptsize] {(1),(3)} (1.85,-2.95);
		\draw[fill=gray!10] (6.2,0) ellipse (1.8 and 0.8);
		\node at (5.6,0.25) {$\hat X$};
		\fill (6.2,-0.8) circle (2.2pt); \node[above] at (6.2,-0.78) {$\hat z$};
		\node at (6.2,-1.35) {a copy of $X$};
		\draw[<->,red,thick] (1.9,0.45) to[bend left=20] node[above,font=\scriptsize] {(1),(2): same move in the other copy} (4.3,0.45);
	\end{tikzpicture}
	\caption{The mirror strategy behind the Twin Lemma: the second player wins $X^{+}\cup\hat X$. (1)~A move inside a pair, or a $z$-quiet move in $X$ or $\hat X$, is copied in the partner. (2)~A move at $z$ or next to $z$ is copied in the other copy of $X$. (3)~A move at the root of $H_i$ is answered at the root of $H_i'$, and condition (C) finishes the argument.}\label{fig:twin}
\end{figure}
\begin{proof}
	We show that $X^+ \cup \hat{X}$ is a second-player win, where $\hat{X}$ is a disjoint copy of $X$ (with $z$ corresponding to $\hat{z}$); then $\G(X^+) \oplus \G(X) = 0$. Call a position \textit{good} if it has the form
	\begin{equation*}
		X'\langle \text{attached pairs } H_i \cong H'_i \rangle \cup \hat{X'} \cup (\text{pairs of isomorphic components}),
	\end{equation*}
	with $X' \in \mathcal{R}(X,z)$ and $\hat{X'} \cong X'$. The starting position is good. We show that every move from a good position has an answer leading to a good position or to a position of Grundy value 0. By induction on the number of vertices, every good position then has Grundy value $0$. The three cases correspond to the arrows in Figure \ref{fig:twin}.
	\begin{enumerate}
		\item A move in a detached pair, in an attached pair (not at its root), or a $z$-quiet move in $X'$ or $\hat X'$: answer with the same move in the partner. The position stays good.
		\item A move at $z$ or at a neighbour of $z$ inside $X'$ (or the corresponding move in $\hat X'$): answer with the corresponding move in the other copy. Afterwards $z$ and $\hat z$ are gone, the attached graphs have become detached pairs, and $X'$, $\hat X'$ have been reduced identically. Everything is paired, so the value is $0$.
		\item A move at the root $v$ of an attached graph $H_i$. This removes $v$ and its neighbours, including $z$, and leaves $X'-z$, the detached piece $H_i-N[v]$, the now detached partner $H_i'$, the other pairs, and $\hat X'$. Answer at the root $v'$ of $H_i'$, which makes $H_i'-N[v']\cong H_i-N[v]$. The remaining value is $\G(X'-z)\oplus\G(\hat X')=\G(X'-z)\oplus\G(X')=0$ by condition (C).
	\end{enumerate}
\end{proof}
\begin{corollary}\label{cor:twin}
	\leavevmode
	\begin{enumerate}
		\item If a vertex $z$ carries an odd number of pendant leaves and, in addition, any number of pendant rooted graphs in isomorphic pairs, the graph has Grundy value $1$.
		\item In particular, $\G(\T{1,2k,s})=1$ for every $k\ge1$ and every finite sequence $s$. For example, $\G(\T{1,2,1^h})=1$ for all $h$: a path with $2h+3$ vertices and one pendant leaf exactly at its middle vertex has value $1$.
	\end{enumerate}
\end{corollary}

\begin{proof}
	Take $X=K_{1,t}$ with centre $z$ and $t$ odd. Every vertex of $X$ is $z$ or a neighbour of $z$, so there are no $z$-quiet moves and $\mathcal R(X,z)=\{X\}$. Moreover $\G(X-z)=t \cdot 1=1$ ($t$ isolated vertices) and $\G(K_{1,t})=1$, so (C) holds, and Lemma~\ref{lem:twin} gives part 1. For part 2, the level-$1$ vertex $z$ of $\T{1,2k,s}$ has the root as its only pendant leaf ($t=1$), and carries $2k$ copies of $\T{s}$, which form $k$ pairs (Figure~\ref{fig:twincor}(a)).
\end{proof}
\begin{figure}
	\centering
	\begin{tikzpicture}[scale=0.85,every node/.style={font=\small}]
		\fill (0,0) circle (2pt) node[right] {root of $\T{1,2k,s}$ (a pendant leaf of $z$)};
		\draw (0,0) -- (0,-0.9);
		\fill (0,-0.9) circle (2.2pt); \node[left] at (-0.05,-0.85) {$z$};
		\foreach \a/\c in {-2.1/blue,-1.2/blue,1.2/orange,2.1/orange} {
			\draw (0,-0.9) -- (\a,-1.9);
			\draw[fill=\c!20] (\a,-1.9) -- (\a-0.45,-3.0) -- (\a+0.45,-3.0) -- cycle;
			\node[font=\scriptsize] at (\a,-2.7) {$\T{s}$};
			\fill (\a,-1.9) circle (1.8pt);}
		\node at (0,-2.3) {$\cdots$};
		\node at (0,-3.45) {$2k$ copies of $\T{s}$, in $k$ pairs};
		\node at (0,-4.2) {(a) $\T{1,2k,s}$: value $1$ for every $k$};
		\begin{scope}[xshift=8.3cm,yshift=-1.2cm]
			\foreach \i in {-6,...,6} \fill (\i*0.45,0) circle (1.8pt);
			\draw (-2.7,0) -- (2.7,0);
			\draw (0,0) -- (0,0.8); \fill (0,0.8) circle (1.8pt) node[above] {root};
			\fill[red] (0,0) circle (2.2pt); \node[below,red] at (0,-0.05) {$z$};
			\draw[decorate,decoration={brace,amplitude=4pt,mirror}] (-2.75,-0.55) -- (-0.4,-0.55) node[midway,below=4pt] {\scriptsize $h+1$};
			\draw[decorate,decoration={brace,amplitude=4pt,mirror}] (0.4,-0.55) -- (2.75,-0.55) node[midway,below=4pt] {\scriptsize $h+1$};
			\node at (0,-3.0) {(b) $\T{1,2,1^h}$: value $1$ for every $h$};
		\end{scope}
	\end{tikzpicture}
	\caption{Corollary~\ref{cor:twin}. (a)~In $\T{1,2k,s}$ the vertex $z$ carries one pendant leaf (the root of the tree) and $k$ pairs of copies of $\T{s}$. (b)~The case $k=1$, $s=(1^h)$.}\label{fig:twincor}
\end{figure}
\section{Long runs of ones: periodicity in the length of a path}\label{sec3}
A \emph{gadget} is a rooted tree $X$ with root $w$. For $h\ge0$ let $X[h]$ be $X$ with a path $v_1v_2\cdots v_h$ (the \emph{handle}) attached by the edge $wv_1$ (Figure~\ref{fig:handle}), and write
\[
t_X(h)=\G(X[h]),\qquad t_X(-1)=\G(X-w).
\]
For instance $\T{1^h,s}=X[h]$ with $X=\T{s}$: the $h$ ones form the handle above the root of $\T{s}$. The \emph{states} of $X$ are the rooted trees obtained from $X$ by moves that keep $w$; there are finitely many. We say that $t_X$ has \emph{period $q$ from $h_0$} if $t_X(h+q)=t_X(h)$ for all $h\ge h_0$.

\begin{figure}
	\centering
	\begin{tikzpicture}[scale=0.8,every node/.style={font=\small}]
		\draw[fill=white] (0,0) -- (-0.9,-1.2) -- (0.9,-1.2) -- cycle; \node at (0,-0.75) {$X$};
		\fill (0,0) circle (2pt) node[above left] {$w$};
		\foreach \i in {1,...,4} { \fill (\i,0) circle (2pt); }
		\draw (0,0) -- (4,0); \draw[dashed] (4,0) -- (5,0); \fill (5,0) circle (2pt); \draw (5,0)--(6,0); \fill (6,0) circle (2pt);
		\node[below] at (1,0) {$v_1$}; \node[below] at (2,0) {$v_2$}; \node[below] at (6,0) {$v_h$};
		\node at (3,0.6) {handle of $h$ vertices};
		\begin{scope}[xshift=8.3cm]
			\draw[fill=white] (0,0) -- (-0.9,-1.2) -- (0.9,-1.2) -- cycle; \node at (0,-0.75) {$X$}; \fill (0,0) circle (2pt);
			\draw (0,0) -- (1,0); \fill (1,0) circle (2pt);
			\node[below] at (0.6,0) {\scriptsize $j{-}2$};
			\draw[gray] (1.6,0) circle (2pt); \draw[gray] (2.2,0) circle (2pt); \draw[gray] (2.8,0) circle (2pt);
			\node[above] at (2.2,0.05) {\scriptsize removed};
			\fill (3.4,0) circle (2pt); \draw (3.4,0) -- (4.6,0); \fill (4.6,0) circle (2pt);
			\node[below] at (4,0) {\scriptsize $h{-}j{-}1$};
			\node at (2.2,-1) {playing $v_j$};
		\end{scope}
	\end{tikzpicture}
	\caption{Left: the gadget $X$ with a handle of $h$ vertices. Right: playing $v_j$ leaves the gadget with a handle of $j-2$ vertices and a separate path of $h-j-1$ vertices.}\label{fig:handle}
\end{figure}
Let $p(k)$ be the Grundy value of the path $P_k$ with $p(-1) = p(0) = 0$.
\begin{lemma}[options of $X\lbrack h\rbrack$]\label{lem:handleoptions}
	For $h\ge1$, the options of $X[h]$ have the following values:
	\begin{enumerate}
		\item[(a)] play $w$: $\ \G(X-N_X[w])\oplus p(h-1)$;
		\item[(b)] play a child $c$ of $w$ in $X$: $\G(X-N_X[c])\oplus p(h)$;
		\item[(c)] play a vertex of $X$ at distance $\ge2$ from $w$: $\ t_{X'}(h)\oplus x$, where $X'$ is the new state of the gadget and $x$ the value of the pieces cut off from it;
		\item[(d)] play $v_1$: $\ \G(X-w)\oplus p(h-2)$;
		\item[(e)] play $v_j$ with $2\le j\le h$: $\ t_X(j-2)\oplus p(h-j-1)$.
	\end{enumerate}
\end{lemma}

\begin{proof}
	Each move deletes a closed neighbourhood. In (a) and (d) the handle loses its first one or two vertices and is cut off from the gadget. In (b) the handle is cut off intact. In (e) the vertices $v_{j-1},v_j,v_{j+1}$ disappear (Figure~\ref{fig:handle}), leaving the gadget with a handle of $j-2$ vertices and a separate path of $h-j-1$ vertices. In (c) the handle is untouched.
\end{proof}

\begin{theorem}\label{thm:handle}
	Let $h_0\ge0$. Suppose that for every state $Y$ of $X$
	\[
	t_Y(h)=t_Y(h-34)\qquad\text{for all }h\text{ with } h_0+34\le h\le h_0+121 .
	\]
	Then $t_Y(h)=t_Y(h-34)$ for all $h\ge h_0+34$ and all states $Y$. In particular, every $t_Y$, and $h\mapsto\G(X[h])$, is periodic from $h_0$ with period 34.
\end{theorem}
\begin{proof}
	We use induction on $h$, and for fixed $h$ on the states (option (c) leads to smaller states). Let $h>h_0+121$, and assume the claim for all smaller $h$ and for smaller states at $h$. We show that $Y[h]$ and $Y[h-34]$ have the same set of option values.
	\begin{itemize}
		\item Options (a), (b), (d) match because $p(n)=p(n-34)$ for $n\ge 86$, and $h-2\ge86$ by Theorem \ref{thm:path}.
		\item Options (c) match by the induction on states.
		\item For options (e), the move $v_j$ in $Y[h]$ cuts the handle into a piece of $j-2$ vertices still attached to the gadget and a separate path of $h-j-1$ vertices. There are two cases.
		\begin{itemize}
			\item \emph{Cut near the gadget} ($j\le h-87$). The separate path is long, $h-j-1\ge 86$, so $p(h-j-1)=p(h-34-j-1)$ by Theorem \ref{thm:path}. The same move $v_j$ in $Y[h-34]$ gives the same value $t_Y(j-2)\oplus p(h-34-j-1)$.
			\item \emph{Cut far from the gadget} ($j\ge h_0+36$). The attached piece is long, $j-2\ge h_0+34$, so by induction $t_Y(j-2)=t_Y(j-36)$. The move $v_{j-34}$ in $Y[h-34]$ gives the same value $t_Y(j-36)\oplus p(h-j-1)$.
		\end{itemize}
		Because $h\ge h_0+122$, the two cases leave no gap (see Figure \ref{fig:handleproof}), so every $j\in[2,h]$ is in one of them. Conversely every $j'\in[2,h-34]$ is matched: $j=j'$ if $j'\le h-87$, and otherwise $j=j'+34\ge h_0+36$.
	\end{itemize}
	Equal option-value sets have equal mex, so $t_Y(h)=t_Y(h-34)$. For $h$ in the window itself the claim is the hypothesis.
\end{proof}
\begin{figure}
	\centering
	\begin{tikzpicture}[every node/.style={font=\small},
		gad/.style={draw,fill=gray!15},
		br/.style={decorate,decoration={brace,amplitude=4pt}},
		brm/.style={decorate,decoration={brace,amplitude=4pt,mirror}}]
		
		\node[anchor=west,font=\normalsize\bfseries] at (-1.2,1.5) {(a) Cut near the gadget: $j\le h-87$ \ (same move $v_j$)};
		\draw[gad] (0,0) -- (-0.7,-0.9) -- (0.7,-0.9) -- cycle; \node at (0,-0.6) {$Y$};
		\fill (0,0) circle (1.8pt);
		\draw[blue,very thick] (0,0) -- (2.0,0);
		\foreach \x in {2.4,2.8,3.2} \draw[gray] (\x,0) circle (2pt);
		\draw[orange,very thick] (3.6,0) -- (11.0,0);
		\fill[red] (2.8,0) circle (2.4pt); \node[red,below] at (2.8,-0.05) {$v_j$};
		\draw[br] (0.1,0.2) -- (2.0,0.2) node[midway,above=4pt,blue] {$j-2$};
		\draw[br] (3.6,0.2) -- (11.0,0.2) node[midway,above=4pt,orange] {$h-j-1\ \ge 86$};
		\node[anchor=east] at (-0.9,0) {$Y[h]$};
		\begin{scope}[yshift=-2.3cm]
			\draw[gad] (0,0) -- (-0.7,-0.9) -- (0.7,-0.9) -- cycle; \node at (0,-0.6) {$Y$};
			\fill (0,0) circle (1.8pt);
			\draw[blue,very thick] (0,0) -- (2.0,0);
			\foreach \x in {2.4,2.8,3.2} \draw[gray] (\x,0) circle (2pt);
			\draw[orange,very thick] (3.6,0) -- (7.6,0);
			\fill[red] (2.8,0) circle (2.4pt); \node[red,below] at (2.8,-0.05) {$v_j$};
			\draw[br] (0.1,0.2) -- (2.0,0.2) node[midway,above=4pt,blue] {$j-2$ (same)};
			\draw[br] (3.6,0.2) -- (7.6,0.2) node[midway,above=4pt,orange] {$h-j-35\ \ge 52$};
			\node[anchor=east] at (-0.9,0) {$Y[h-34]$};
		\end{scope}
		\node[align=left,anchor=west] at (8.0,-2.9) {value $t_Y(j-2)\oplus p(h-j-1)$\\ $=t_Y(j-2)\oplus p(h-j-35)$\\ since $p$ is $34$-periodic from $52$};
		
		\begin{scope}[yshift=-6.4cm]
			\node[anchor=west,font=\normalsize\bfseries] at (-1.2,1.5) {(b) Cut far from the gadget: $j\ge h_0+36$ \ (move $v_j\ \leftrightarrow\ v_{j-34}$)};
			\draw[gad] (0,0) -- (-0.7,-0.9) -- (0.7,-0.9) -- cycle; \node at (0,-0.6) {$Y$};
			\fill (0,0) circle (1.8pt);
			\draw[blue,very thick] (0,0) -- (7.4,0);
			\foreach \x in {7.8,8.2,8.6} \draw[gray] (\x,0) circle (2pt);
			\draw[orange,very thick] (9.0,0) -- (11.0,0);
			\fill[red] (8.2,0) circle (2.4pt); \node[red,below] at (8.2,-0.05) {$v_j$};
			\draw[br] (0.1,0.2) -- (7.4,0.2) node[midway,above=4pt,blue] {$j-2\ \ge h_0+34$};
			\draw[br] (9.0,0.2) -- (11.0,0.2) node[midway,above=4pt,orange] {$h-j-1$};
			\node[anchor=east] at (-0.9,0) {$Y[h]$};
			\begin{scope}[yshift=-2.3cm]
				\draw[gad] (0,0) -- (-0.7,-0.9) -- (0.7,-0.9) -- cycle; \node at (0,-0.6) {$Y$};
				\fill (0,0) circle (1.8pt);
				\draw[blue,very thick] (0,0) -- (4.0,0);
				\foreach \x in {4.4,4.8,5.2} \draw[gray] (\x,0) circle (2pt);
				\draw[orange,very thick] (5.6,0) -- (7.6,0);
				\fill[red] (4.8,0) circle (2.4pt); \node[red,below] at (4.8,-0.05) {$v_{j-34}$};
				\draw[br] (0.1,0.2) -- (4.0,0.2) node[midway,above=4pt,blue] {$j-36\ \ge h_0$};
				\draw[br] (5.6,0.2) -- (7.6,0.2) node[midway,above=4pt,orange] {$h-j-1$ (same)};
				\node[anchor=east] at (-0.9,0) {$Y[h-34]$};
			\end{scope}
			\node[align=left,anchor=west] at (8.0,-2.9) {value $t_Y(j-2)\oplus p(h-j-1)$\\ $=t_Y(j-36)\oplus p(h-j-1)$\\ since $t_Y$ is $34$-periodic from $h_0$};
		\end{scope}
		
		\begin{scope}[yshift=-12.6cm]
			\node[anchor=west,font=\normalsize\bfseries] at (-1.2,2.0) {(c) Every move $v_j$, $2\le j\le h$, is in case (a) or (b) once $h\ge h_0+122$};
			\draw[-{Stealth}] (0,0) -- (11.6,0) node[right] {$j$};
			\foreach \x/\l in {0.3/2, 11.0/h} {\draw (\x,0.1) -- (\x,-0.1) node[below] {$\l$};}
			\draw (7.0,0.1) -- (7.0,-0.1) node[below] {$h-87$};
			\draw (4.2,0.1) -- (4.2,-0.1) node[below] {$h_0+36$};
			\draw[blue!70,line width=4pt,opacity=0.6] (0.3,0.35) -- (7.0,0.35);
			\node[blue!70!black,above] at (2.2,0.45) {case (a): $j\le h-87$};
			\draw[green!50!black,line width=4pt,opacity=0.6] (4.2,0.75) -- (11.0,0.75);
			\node[green!40!black,above] at (9.0,0.85) {case (b): $j\ge h_0+36$};
			\draw[brm] (4.2,-0.65) -- (7.0,-0.65) node[midway,below=5pt,align=center] {no gap between the cases iff\\ $h-87\ge h_0+35$, i.e. $h\ge h_0+122$};
		\end{scope}
	\end{tikzpicture}
	\caption{The matching of options (e) in the proof of Theorem~\ref{thm:handle}. (a)~If the cut is near the gadget ($j\le h-87$), the move $v_j$ in $Y[h]$ is matched with the same move in $Y[h-34]$: the attached piece is unchanged and the separate path is shortened by $34$, which does not change its value because it keeps at least $52$ vertices. (b)~If the cut is far from the gadget ($j\ge h_0+36$), $v_j$ is matched with $v_{j-34}$: the separate path is unchanged and the attached piece is shortened by $34$, which does not change $t_Y$ by induction. (c)~For $h\ge h_0+122$ the two cases leave no gap, so every move is matched.}\label{fig:handleproof}
\end{figure}

\begin{theorem}\label{thm:boundedperiodic}
	For a gadget $X$ the following are equivalent:
	\begin{enumerate}
		\item the values $t_{Y}(h)$ (where $Y$ ranges over the states of $X$ and $h\ge0$) are bounded;
		\item every sequence $h\mapsto t_{Y}(h)$ is eventually periodic.
	\end{enumerate}
\end{theorem}

\begin{proof}
	$(2) \Rightarrow (1)$ is clear. $(1)\Rightarrow(2)$: Let $M$ bound all values. We use induction over the states, in order of the number of vertices; a $w$-quiet move always leads to a state with fewer vertices. Fix a state $Y$ and assume that all smaller states are eventually periodic. Choose a common period $\Pi$ (a multiple of $34$) and a threshold $H_1$ such that, for $h\ge H_1$, the options (a), (b), (c), (d) of $Y[h]$ depend only on $h\bmod\Pi$. This is possible because $p$ is periodic from $52$ (Theorem~\ref{thm:path}) and, by induction, so is every smaller state.
	
	For option (e), playing $v_j$ gives $t_Y(r)\oplus p(h-r-3)$, where $r=j-2$ and the separate path has $h-r-3$ vertices. We split the cases by the length of the separate path:
	\begin{enumerate}
		\item \textbf{Long separate path:} $h-r-3\ge 52$, i.e.\ $r\le h-55$.\\
		By Theorem~\ref{thm:path}, $p(h-r-3)$ depends only on $(h-r)\bmod 34$. Group the values of $r$ by their residue $\rho=r\bmod 34$, and let
		\begin{equation*}
			V_\rho(h)=\{t_Y(r): r\le h-55 \text{ and } r\equiv\rho \pmod{34}\}.
		\end{equation*}
		For $h\ge 88$ we have $h-3-\rho\ge 52$, so this part of the option set is
		\begin{equation*}
			\bigcup_{\rho=0}^{33}\big(V_\rho(h)\oplus p(h-3-\rho)\big).
		\end{equation*}
		Each $V_\rho(h)$ grows with $h$ but stays inside $\{0,1,\dots,M\}$, so it changes only finitely many times. Hence there is $H_2$ such that all the sets $V_\rho(h)$ are constant for $h\ge H_2$, and from then on this part of the option set depends only on $h\bmod 34$.
		\item \textbf{Short separate path:} $h-54\le r\le h-2$.\\
		Here the separate path has at most $51$ vertices (it is empty when $r=h-2$), and its value is one of the fixed numbers $p(-1)=p(0)=0,\,p(1),\dots,p(51)$. So this part of the option set is determined by $t_Y(h-54),t_Y(h-53),\dots,t_Y(h-2)$.
	\end{enumerate}
	Hence, for $h\ge H:=\max\{H_1,H_2,88\}$,
	\begin{equation*}
		t_Y(h)=F\big(h\bmod\Pi,\;t_Y(h-54),t_Y(h-53),\dots,t_Y(h-2)\big)
	\end{equation*}
	for one fixed function $F$. Now define the state vector
	\begin{equation*}
		s_h=\big(h\bmod\Pi,\;t_Y(h-54),t_Y(h-53),\dots,t_Y(h-1)\big).
	\end{equation*}
	For $h\ge H$, the vector $s_{h+1}$ is determined by $s_h$ as follows:
	\begin{enumerate}
		\item \textbf{residue:} $h\bmod\Pi$ is replaced by $(h+1)\bmod\Pi$;
		\item \textbf{new value:} $t_Y(h)=F\big(h\bmod\Pi,\;t_Y(h-54),\dots,t_Y(h-2)\big)$ is computed from $s_h$;
		\item \textbf{window:} $t_Y(h-54)$ is dropped and $t_Y(h)$ is appended.
	\end{enumerate}
	So $s_{h+1}=\Phi(s_h)$ for a fixed map $\Phi$. Every $s_h$ lies in a set of size at most $N:=\Pi(M+1)^{54}$. By the pigeonhole principle, two of the $N+1$ vectors $s_H,s_{H+1},\dots,s_{H+N}$ coincide, say $s_a=s_b$ with $H\le a<b$. Since $s_{h+1}=\Phi(s_h)$, it follows that $s_{a+k}=s_{b+k}$ for every $k\ge0$. The last entry of $s_{h+1}$ is $t_Y(h)$, so
	\begin{equation*}
		t_Y\big(h+(b-a)\big)=t_Y(h)\qquad\text{for all } h\ge a.
	\end{equation*}
	Thus $t_Y$ is eventually periodic, with period $b-a\le N$, from $h=a$ on.
\end{proof}

For two gadgets $X,Y$ joined by a path $w_X v_1\cdots v_m w_Y$ (Figure~\ref{fig:twoended}), let $u_{X,Y}(m)$ be the Grundy value. A move on the path splits it into two one-ended pieces, with value $t_X(j-2)\oplus t_Y(m-j-1)$, so no new finite check is needed.

\begin{figure}
	\centering
	\begin{tikzpicture}[scale=0.8,every node/.style={font=\small}]
		\draw[fill=white] (0,0) -- (-0.9,-1.2) -- (0.9,-1.2) -- cycle; \node at (0,-0.75) {$X$};
		\fill (0,0) circle (2pt) node[above] {$w_X$};
		\draw (0,0) -- (2.5,0); \draw[dashed] (2.5,0) -- (3.5,0); \draw (3.5,0) -- (6,0);
		\foreach \x in {1,2,4,5} \fill (\x,0) circle (2pt);
		\node[below] at (1,0) {$v_1$}; \node[below] at (5,0) {$v_m$};
		\draw[fill=white] (6,0) -- (5.1,-1.2) -- (6.9,-1.2) -- cycle; \node at (6,-0.75) {$Y$};
		\fill (6,0) circle (2pt) node[above] {$w_Y$};
		\node at (3,0.55) {$m$ path vertices};
	\end{tikzpicture}
	\caption{Two gadgets joined by a path. For $X=Y=\T{s}$ and $m=2h+1$ this is the level-regular tree $\T{2,1^h,s}$.}\label{fig:twoended}
\end{figure}
\begin{theorem}[two gadgets]\label{thm:twoended}
	Suppose that $t_{X'}$ is periodic with period 34 from $h_X$ for every state $X'$ of $X$, and $t_{Y'}$ is periodic with period 34 from $h_Y$ for every state $Y'$ of $Y$. Then for all states $X',Y'$, $u_{X',Y'}(m+34)=u_{X',Y'}(m)$ for every $m\ge h_X+h_Y+36$.
\end{theorem}

\begin{proof}
	Fix $m\ge h_X+h_Y+36$ and compare the options of the positions with $m+34$ and $m$ path vertices.  We use induction over pairs of states $(X',Y')$, in order of the total number of vertices. A move at a child $c$ of $w_X$ in $X'$ gives $\G(X'-N_{X'}[c])\oplus t_{Y'}(m+34)$ and $\G(X'-N_{X'}[c])\oplus t_{Y'}(m)$, which are equal since $m \ge h_{Y}$. Playing $w_X$ gives $\G(X'-N_{X'}[w_{X}])\oplus t_{Y'}(m+33)$ and $\G(X'-N{X'}[w_X])\oplus t_{Y'}(m-1)$. These are equal because $m-1\ge h_Y$. The moves at $w_Y$ and at its children are symmetric. A quiet move inside $X'$ or $Y'$ leads to a pair of smaller states with the same number of path vertices, so it matches by induction. A move at $v_j$ ($1\le j\le m+34$) gives $t_{X'}(j-2)\oplus t_{Y'}(m+33-j)$. If $j\le m-h_Y-1$, then $t_{Y'}(m-j+33)=t_{Y'}(m-j-1)$, and the same $j$ works for $m$ path vertices. If $j\ge h_X+36$, then $t_{X'}(j-2)=t_{X'}(j-36)$, and $j-34$ works for $m$ path vertices. Both cases cover $[1,m+34]$ when $m-h_Y-1 \ge h_X + 35$, which is $m\ge h_X+h_Y+36$. Conversely, every move $v_{j'}$ ($1\le j'\le m$) of the shorter position is matched: by $j=j'$ if $j'\le m-h_Y-1$, and otherwise by $j=j'+34$, since then $j'\ge m-h_Y\ge h_X+36$, so $j\ge h_X+36$. Hence the two positions have the same set of option values, and therefore the same mex: $u_{X',Y'}(m+34)=u_{X',Y'}(m)$.
\end{proof}
\section{Concluding remarks and conjectures}
Extensive computer searches suggest that the Grundy values are eventually periodic both in the number of branches at a fixed level and in the length of a run of ones. We therefore propose the following conjectures.

\begin{conjecture}\label{conj:eventual}
	For all $a_1,\dots,a_{i-1}$ and $a_{i+1},\dots,a_d$, the sequence
	\[
	n\mapsto\G\left(\T{a_1,\dots,a_{i-1},n,a_{i+1},\dots,a_d}\right)
	\]
	is eventually periodic with period $2$.
\end{conjecture}

Conjecture~\ref{conj:eventual} holds when $d\le2$, by Theorems~\ref{thm1} and~\ref{thm2}, and for the last level $i=d$, by Corollary~\ref{cor:pair}. Corollary~\ref{cor:twin} gives further partial evidence: for $a_1=1$ and every suffix $s$, the sequence $n\mapsto\G(\T{1,n,s})$ is constantly equal to $1$ on even values of $n$.

For a rooted graph $(H,v)$, a graph $G$ and a vertex $u$ of $G$, let $G_u\langle H;k\rangle$ denote the graph obtained from $G$ by attaching $k$ disjoint copies of $H$, each joined to $u$ by an edge from its root $v$. A more general form of Conjecture~\ref{conj:eventual} is the following.

\begin{conjecture}\label{conj:lep}
	Let $H$ be a level-regular rooted tree. Then for every graph $G$ and every vertex $u$ of $G$, the sequence $k\mapsto\G\left(G_u\langle H;k\rangle\right)$ is eventually periodic with period $2$.
\end{conjecture}

For runs of ones, we conjecture the following.

\begin{conjecture}\label{conj:runs}
	For every finite sequence $s$, the sequences $h\mapsto\G\left(\T{1^h,s}\right)$ and $h\mapsto\G\left(\T{s,1^h}\right)$ are eventually periodic, with period dividing $34$.
\end{conjecture}

For $\T{1^h,s}$, Theorem~\ref{thm:boundedperiodic} shows that Conjecture~\ref{conj:runs} is equivalent to the boundedness of the Grundy values, and Theorem~\ref{thm:handle} reduces each individual case to a finite check. For $\T{s,1^h}$ with $s=(a_1,\dots,a_m)$, the conjecture holds when $a_1\cdots a_m=1$, since the tree is then a path; when $s=(2,1^b)$, which again gives a path; and when $s=(1,2k,s'')$, since the value is then always $1$ by Corollary~\ref{cor:twin}. The conjecture is open as soon as there are three long legs, starting with $\T{3,1^h}$.
\section*{Acknowledgement} 
This project is funded by National Research Council of Thailand (NRCT) : Contract number N42A690178 and was partially supported by the Faculty of Science at Sriracha, Kasetsart University (Grant No. Sci-Src-M01/2569). The author gratefully acknowledges the financial support provided by these funding sources.

The author would also like to express sincere gratitude to Pawaton Kaemawichanurat, the research advisor, for his valuable guidance, insightful advice, and continuous support throughout this project.

\section*{Declaration of generative AI and AI-assisted technologies in the manuscript preparation process} 
During the preparation of this manuscript, the author used Claude to assist with grammar and language editing, the development of code for extensive computational searches. The author carefully reviewed and revised all AI-assisted content as necessary and takes full responsibility for the accuracy, integrity, and content of the final manuscript.
\printcredits

\bibliographystyle{cas-model2-names}

\bibliography{cas-refs}

@article{brown2020nimber,
  title={Nimber sequences of node-kayles games},
  author={Brown, Sierra and Daugherty, S and Fiorini, E and Maldonado, B and Manzano-Ruiz, D and Rainville, S and Waechter, R and Wong, T},
  journal={Journal of integer sequences},
  volume={23},
  year={2020}
}

@article{guignard2009compound,
  title={Compound node--kayles on paths},
  author={Guignard, Adrien and Sopena, {\'E}ric},
  journal={Theoretical Computer Science},
  volume={410},
  number={21-23},
  pages={2033--2044},
  year={2009},
  publisher={Elsevier}
}

@article{bodlaender2015exact,
  title={Exact algorithms for Kayles},
  author={Bodlaender, Hans L and Kratsch, Dieter and Timmer, Sjoerd T},
  journal={Theoretical Computer Science},
  volume={562},
  pages={165--176},
  year={2015},
  publisher={Elsevier}
}

@article{bodlaender2002kayles,
  title={Kayles and nimbers},
  author={Bodlaender, Hans L and Kratsch, Dieter},
  journal={Journal of Algorithms},
  volume={43},
  number={1},
  pages={106--119},
  year={2002},
  publisher={Elsevier}
}

@book{conway2000numbers,
  title={On numbers and games},
  author={Conway, John H},
  year={2000},
  publisher={AK Peters/CRC Press}
}

@book{berlekamp2001winning,
  title={Winning ways for your mathematical plays, volume 1},
  author={Berlekamp, Elwyn R and Conway, John H and Guy, Richard K},
  year={2001},
  publisher={AK Peters/CRC Press}
}

@book{berlekamp2003winning,
  title={Winning ways for your mathematical plays, volume 2},
  author={Berlekamp, Elwyn R and Conway, John H and Guy, Richard K},
  year={2003},
  publisher={AK Peters/CRC Press}
}

@book{albert2019lessons,
  title={Lessons in play: an introduction to combinatorial game theory},
  author={Albert, Michael and Nowakowski, Richard and Wolfe, David},
  year={2019},
  publisher={AK Peters/CRC Press}
}

@article{sprague1935mathematische,
  title={{\"U}ber mathematische kampfspiele},
  author={Sprague, Richard},
  journal={Tohoku Mathematical Journal, First Series},
  volume={41},
  pages={438--444},
  year={1935},
  publisher={Mathematical Institute, Tohoku University}
}

@article{grundy1939mathematics,
  title={Mathematics and games},
  author={Grundy, Patrick M},
  journal={Eureka},
  volume={2},
  pages={6--8},
  year={1939}
}

@article{duchene2009combinatorial,
  title={Combinatorial graph games},
  author={Duchene, Eric and Gravier, Sylvain and Mhalla, Mehdi},
  journal={Ars Combin},
  volume={90},
  pages={33--44},
  year={2009}
}

@inproceedings{fleischer2004kayles,
  title={Kayles on the way to the stars},
  author={Fleischer, Rudolf and Trippen, Gerhard},
  booktitle={International Conference on Computers and Games},
  pages={232--245},
  year={2004},
  organization={Springer}
}

@book{dudeney2002canterbury,
  title={The Canterbury Puzzles},
  author={Dudeney, Henry Ernest},
  year={2002},
  publisher={Courier Corporation}
}

@book{loyd1976sam,
  title={Sam Loyd's cyclopedia of 5000 puzzles, tricks and conundrums with answers},
  author={Loyd, Sam},
  year={1976},
  publisher={The Lamb Publishing Company}
}

@book{dawson1935caissa,
  title={Caissa's Wild Roses},
  author={Dawson, Thomas Rayner},
  number={1},
  year={1935},
  publisher={TR Dawson}
}

@misc{flammenkamp2000sprague,
 author = {Flammenkamp, Achim},
 title = {Sprague-Grundy Values of Octal-Games},
year = {2000},
note = {Last updated 05 May 2021}, 
 url = {https://wwwhomes.uni-bielefeld.de/achim/octal.html}
}


\end{document}